\documentclass{article}
\usepackage[dvipdfmx]{graphicx}
\usepackage{amsmath}
\usepackage{amsthm}
\usepackage{amsfonts}
\usepackage{amssymb}
\usepackage{latexsym}
\usepackage{color}
\def\qed{\hfill $\Box$}
\makeatletter
\renewenvironment{proof}[1][\proofname]{\par
  \normalfont
  \topsep6\p@\@plus6\p@ \trivlist
  \item[\hskip\labelsep{\bfseries #1}\@addpunct{\bfseries.}]\ignorespaces
}{
  \endtrivlist
}
\renewcommand{\proofname}{Proof}
\makeatother
\begin{document}
\theoremstyle{plain}
\newtheorem{defi}{Definition}[section]
\newtheorem{lem}[defi]{Lemma}
\newtheorem{thm}[defi]{Theorem}
\newtheorem{rem}[defi]{Remark}
\newtheorem{fact}[defi]{Fact}
\newtheorem{ex}[defi]{Example}
\newtheorem{prop}[defi]{Proposition}
\newtheorem{cor}[defi]{Corollary}
\newtheorem{condition}[defi]{Condition}
\newtheorem{assumption}[defi]{Assumption}
\newcommand{\sign}{\mathop{\rm sign}}
\newcommand{\conv}{\mathop{\rm conv}}
\newcommand{\argmax}{\mathop{\rm arg~max}\limits}
\newcommand{\argmin}{\mathop{\rm arg~min}\limits}
\newcommand{\argsup}{\mathop{\rm arg~sup}\limits}
\newcommand{\arginf}{\mathop{\rm arg~inf}\limits}
\newcommand{\diag}{\mathop{\rm diag}}
\newcommand{\minimize}{\mathop{\rm minimize}\limits}
\newcommand{\maximize}{\mathop{\rm maximize}\limits}
\title{The Dantzig selector for a linear model of diffusion processes}
\author{Kou Fujimori
\\
\\
Waseda University}
\date{}
\maketitle
\begin{abstract}
In this paper, a linear model of diffusion processes with unknown 
drift and diagonal diffusion matrices is discussed.  
We will consider the estimation problems for unknown parameters based on the discrete time observation in high-dimensional and sparse settings. 
To estimate drift matrices, the Dantzig selector which was 
proposed by Cand\'es and Tao in 2007 will be applied.
Then, we will prove two types of consistency of the estimator of drift matrix; one is the consistency in the sense of $l_q$ norm for every $q \in [1,\infty]$ and the other is the variable selection consistency.
Moreover, we will construct an asymptotically normal estimator of the 
drift matrix by using the variable selection consistency of the Dantzig selector. 
\end{abstract}
\section{Introduction}
Let us consider the following 
model given by the linear stochastic differential equation: 
\begin{equation}\label{model}
X_t = X_0 + \int_0^t \Theta^T \phi(X_s) ds + \sigma W_t, 
\end{equation}
where $\{X_t\}_{t \geq 0} = \{(X_t^1,\ldots,X_t^p)\}_{t \geq 0}$ is a $p$-dimensional process, 
$\{W_t\}_{t \geq 0} := \{(W_t^1,\ldots,W_t^p)\}_{t \geq 0}$ be a $p$-dimensional standard Brownian motion,
$\Theta$ is a $p\times p$ sparse deterministic matrix, and 
$\sigma = \diag(\sigma_1,\ldots,\sigma_p)$ is a $p \times p$ diagonal matrix and 
$\phi(x) = (\phi_1(x_1), \ldots, \phi_p(x_p))$ for $x = (x_1, \ldots, x_p) \in \mathbb{R}^p$ is a smooth 
$\mathbb{R}^p$-valued function. 
We will propose some estimators for the drift matrix $\Theta$ and 
the diffusion matrix $\sigma$ based on the
observation of $\{X_t\}_{t \geq 0}$ at $n+1$ equidistant time points 
$0 =:t_0^n < t_1^n < \ldots < t_n^n$, under the high-dimensional and sparse setting, $i.e.$, 
$p \gg n$ and the number of nonzero components of the true value $\Theta^0$
is relatively small. 

To deal with such high-dimensional and sparse parameters, various kinds of estimators for 
regression models have been discussed. 
One of the most famous estimation methods is the $l_1$-penalized method called 
Lasso proposed by 
\cite{key tibshirani96}, which has been 
studied for regression models with high-dimensional and sparse parameters 
including the models of stochastic processes. 
On the other hand, the relatively new estimation procedure called the Dantzig selector 
was proposed for linear regression models in \cite{key tao}.
The Dantzig selector has some properties similar to Lasso estimator for linear regression models in some theoretical senses \cite{key bickel-ritov-tsybakov}.
Moreover, it is well known that the Dantzig selector for linear models 
has computational advantages since it can be solved by a linear programming, while Lasso demands
a convex program.

The estimation problems for models of diffusion processes 
based on discretely observed data have been studied by many researchers in 
low dimensional settings.
Especially, the quasi-likelihood methods have been used to estimate the unknown parameter, for instance, 
see 
\cite{yoshida92}, 
\cite{key jacod}, and 
\cite{key kessler97}. 
In addition, the penalized estimators for discretely observed multi-dimensional models of diffusion processes were discussed by \cite{iacus} and \cite{key masuda-shimizu} in low-dimensional settings.

Pioneering work of high-dimensional linear diffusion processes was done by 
\cite{periera-ibrahimi}.
They studied the various models of multi-dimensional diffusion processes observed continuously in high-dimensional settings including the 
following $p$-dimensional linear models:
\begin{equation}\label{periera}
X_t = X_0 + \int_0^t \Theta^T X_s ds + W_t.
\end{equation}
These models may be useful for various fields such as statistical physics, chemical reactions and network systems. 
They proposed the Lasso type estimator of the drift matrix $\Theta$ and prove the variable selection consistency of the estimator. 
In this paper, we will apply the Dantzig selector for the 
linear models of stochastic processes (\ref{model}), which is similar to the model (\ref{periera}), to estimate the drift matrix $\Theta$ 
and prove the consistency in the sense of $l_q$ norm for every $q \in [1,\infty]$ and the
variable selection consistency under some appropriate conditions.
Moreover, using the variable selection consistency, we will construct the new estimator 
which has an asymptotic normality.

This paper is organized as follows. 
In Section 2, we will introduce our model setups and some regularity conditions. 
The construction and the consistency of the estimator of diffusion matrix are described in Section 3.
The estimation procedure for the drift matrix and the $l_q$ consistency of the estimator are presented in Section 4. 
Then, we will prove the variable selection consistency in Section 5. 
Moreover, we will construct the new estimator by using the variable selection consistency 
of the drift estimator and prove the asymptotic normality of the new estimator in this section. Finally, some concluding remarks and future works are described in Section 6.

Throughout this paper, we denote by $\|\cdot\|_q$ the $l_q$ norm of vector for every $q \in[1,\infty]$, $i.e.$ for $v = (v_1,v_2,\ldots,v_p)^T \in \mathbb{R}^p$, we define that 
\begin{align*}
\|v\|_q &= \left(\sum \limits_{j=1}^p |v_j|^q \right)^{\frac{1}{q}},\quad q < \infty ; \\
\|v\|_{\infty} &= \sup \limits_{1 \leq j \leq p} |v_j|.
\end{align*}
In addition, for a $m \times n$ matrix $A$, where $m,\ n \in \mathbb{N}$, we define $\|A\|_{\infty}$ by
\[\|A\|_{\infty} := \sup \limits_{1 \leq i \leq m} \sup \limits_{1 \leq j \leq n} |A_i^j|,\]
where $A_i^j$ denotes the $(i,j)$-component of the matrix $A$.
For a vector $v \in \mathbb{R}^p$, and an index set $T\subset \{1,2,\ldots,p\}$, 
we write $v_T$ for the $|T|$-dimensional sub-vector of $v$ restricted by the index set $T$, 
where $|T|$ is the number of elements in the set $T$.
Similarly, for a $p \times p$ matrix $A$ and index sets $T,T' \subset \{1,2,\ldots,p\}$, 
we define the $|T| \times |T'| $ sub-matrix $A_{T,T'}$ by 
\[
A_{T,T'} := (A_{i,j})_{i \in T, j \in T'}.
\]
For a $\mathbb{R}$-valued random variable $X$ on a probability space $(\Omega, \mathcal{F}, P)$, we define the $L_q$ norm of $X$ by 
\[
\|X\|_{L^q} := (E[|X|^q])^{\frac{1}{q}},
\]
where $E[\cdot]$ is the expectation with respect to the probability measure $P$.
\section{Preliminaries}
Let $\{W_t^1\}_{t \geq 0},\{W_t^2\}_{t \geq0},\ldots$ be independent
standard Brownian motions on a probability space $(\Omega,\mathcal{F},P)$.
Define the filtration $\{\mathcal{F}_t\}_{t \geq 0}$ as follows.
\[\mathcal{F}_t := \mathcal{F}_0 \lor \sigma(W_s^j,\ j = 1,2,\ldots : s \in [0,t]),\quad t \geq0,\]
where $\mathcal{F}_0$ is a $\sigma$-field independent of $\{W_t\}_{t\geq0}$.
We consider the following $p$-dimensional linear
stochastic differential equation (\ref{model}) defined on the stochastic basis
$(\Omega,\mathcal{F},\{\mathcal{F}_t\}_{t\geq0},P)$:
\begin{equation*}
X_t = X_0 + \int_0^t \Theta^T \phi(X_s) ds + \sigma W_t, 
\end{equation*}
where $\{X_t\}_{t \geq 0} = \{(X_t^1,\ldots,X_t^p)\}_{t \geq 0}$ is a $p$-dimensional process, 
$\{W_t\}_{t \geq 0} := \{(W_t^1,\ldots,W_t^p)\}_{t \geq 0}$ be a $p$-dimensional standard Brownian motion,
$\Theta$ is a $p\times p$ deterministic matrix, and 
$\sigma = \diag(\sigma_1,\ldots,\sigma_p)$ is a $p \times p$ diagonal matrix and 
$\phi(x) = (\phi_1(x_1), \ldots, \phi_p(x_p)), x = (x_1, \ldots, x_p)$ is a smooth 
$\mathbb{R}^p$-valued function. 
Assume that $X_0$ is $\mathcal{F}_0$-measurable.
Note that $\{X_t^i\}_{t \geq0}$ for each $i = 1,2,\ldots,p$ satisfies the following equation.
\[X_t^i = X_0^i + \int_0^t \Theta_i^T \phi(X_s) ds + \sigma_i W_t^i,\]
where $\Theta_i$ is an $i$-th row of matrix $\Theta$. 
In this paper, we consider the estimation problem of $\Theta$ and $\sigma$. 
We observe the process $\{X_t\}_{t \geq 0}$ at $n+1$ discrete time points: 
\[0 =:t_0^n < t_1^n < \ldots < t_n^n,\quad t_k^n = \frac{k t_n^n}{n},\quad k= 0,1,\ldots,n.\]
Write $T_0^i$ for the support of $\Theta_i^0$ for every $i \in \{1,2,\ldots,p\}$, $i.e.$,
$T_0^i = \{j : \Theta_{ij}^0 \not = 0\}$. Let $S_i$ be the number of elements in the index set $T_0^i$.
Hereafter, we assume the following high-dimensional and sparse setting for the true value $\Theta^0$.
\[p=p_n \gg n,\quad  \sup_{1 \leq i < \infty} S_i =: S^* \ll n,\]
where $S^*>0$ is a constant which does not depend on $n$.
We introduce the $\log$-quasi-likelihood given by 
\[l_n(\Theta_i, \sigma_i) := \frac{1}{n \Delta_n} \sum_{k=1}^n \left\{-\frac{1}{2}\log(2\pi \sigma^2 \Delta_n)
-\frac{|X_{t_k^n}^i-X_{t_{k-1}^n}-\Theta_i^T X_{t_{k-1}^n}\Delta_n|^2}{2\sigma_i^2 \Delta_n}\right\},\]
where $\Delta_n := t_k^n - t_{k-1}^n = t_n^n/n$. 
We assume the following conditions.
\begin{assumption}\label{regularity}
\begin{description}
\item[$(i)$]
Suppose that $p_n \rightarrow \infty$ as $n \rightarrow \infty$, $\log p_n = o(\sqrt{n \Delta_n})$, and $\Delta_n \asymp n^{-\alpha}$, for some $\alpha \in(1/2,1)$. 
Especially, the last condition implies that $n \Delta_n = T_n \rightarrow \infty$, 
$\Delta_n \rightarrow 0$ and $n\Delta_n^2 \rightarrow 0$ as $n \rightarrow \infty$. 
\item[$(ii)$]
The functions $\phi_i$'s are uniformly bounded and 
satisfy global Lipschitz condition, $i.e.$, there exist positive constants $L$ and $L'$ 
which satisfy the following conditions for every 
$x,y \in \mathbb{R}$.
\[
\sup_{1 \leq i < \infty} \sup_{x \in \mathbb{R}} |\phi_i(x)| \leq L
\]
\[
\sup_{1 \leq i < \infty} |\phi_i(x) - \phi_i(y)| \leq L' |x-y|.
\]
\item[$(iii)$]
For every $\nu \geq 1$, there exists $\tilde{C}_\nu$ such that 
\[\sup_{1 \leq i < \infty}\sup_{t \in [0,\infty)} E\left[ |X_t^i|^\nu \right] \leq \tilde{C}_\nu.\]
Note that this assumption implies that
\[\sup_{t \in [0,\infty)} E\left[\sup_{1 \leq i \leq p_n} |X_t^i|^\nu \right] \leq p_n \tilde{C}_\nu.\]
\item[$(iv)$]
There exist positive constants $K_1,\ K_2,\ K_3$ and $K_4$ for the true values $\Theta^0$, $\sigma^0$
such that 
\[
K_2 < \inf_{1 \leq i <\infty,\ j \in T_0^i} |\Theta_{ij}^0| \leq \sup_{1\leq i,j <\infty} |\Theta_{ij}^0| < K_1,
\]
\[K_4 < \inf_{1 \leq i <\infty} |\sigma_i^0| \leq
\sup_{1 \leq i <\infty} |\sigma_i^0| < K_3.
\]
\item[$(v)$]
The $\mathbb{R}^{S_i}$-valued process $\{X_{tT_0^i}\}_{t \in [0,T_n]}$ is ergodic for $\Theta = \Theta^0$, 
$\sigma = \sigma^0$ and every $i \in \mathbb{N}$ with invariant measure $\mu_0^i$.
\end{description}
\end{assumption}
Assumption $(iii)$ is satisfied if $X_0$ is a Gaussian random variable. 
In particular, for the process $\{X_t\}_{t \geq 0}$ which satisfies (\ref{model}), 
Assumption $(iii)$ implies that  
for every $\nu \geq 1$, there exists a constant $C_\nu>0$ such that 
for all $n$, $i = 1,2,\ldots,p_n$ and $k = 1,2,\ldots,n$, 
\[E\left[\sup_{s \in [t_{k-1}^n,t_k^n]} |X_s^i - X_{t_{k-1}^n}^i|^\nu\right] \leq C_\nu \Delta_n^{\frac{\nu}{2}}.\]
\section{Estimators for diffusion coefficients}
It is well known that we can ignore the influence of $\Theta$ when we estimate the diffusion coefficients $\sigma$.
So we take $\Theta = 0$ and define the estimator of $\sigma_i$ by the solution $\hat{\sigma}_{n,i}$ to 
the following equation : 
\[\frac{\partial}{\partial \sigma_i} l_n(0,\sigma_i) = 0.\]
Note that $\hat{\sigma}_{n,i}$ can be written explicitly by the following form: 
\[\hat{\sigma}_i^2 :=\hat{\sigma}_{n,i}^2 =  \frac{1}{n \Delta_n} \sum_{k=1}^n |X^i_{t_k^n}-X^i_{t_{k-1}^n}|^2.\]
The next theorem states the consistency of $\hat{\sigma}_i$ uniformly in $i$.
\begin{thm}\label{diffusion}
Under Assumption \ref{regularity}, it holds that 
\[\sup_{1 \leq i \leq p_n} |\hat{\sigma}_i^2 -(\sigma_i^0)^2| \rightarrow^p 0, \quad n\rightarrow \infty.\]
\end{thm}
\begin{proof}
It is clear that
\begin{eqnarray*}
\hat{\sigma}_i^2 
&=& \frac{1}{n \Delta_n} \sum_{k=1}^n 
\left|
\int_{t_{k-1}^n}^{t_k^n} (\Theta_i^0)^T \phi(X_s) ds + \sigma_i^0 (W_{t_k^n}- W_{t_{k-1}^n}) 
\right|^2 \\
&=& \frac{1}{n \Delta_n} \sum_{k=1}^n \left| \int_{t_{k-1}^n}^{t_k^n} (\Theta_i^0)^T \phi(X_s) ds \right|^2
+ \frac{2 \sigma_i^0}{n \Delta_n} \sum_{k=1}^n 
\left(
\int_{t_{k-1}^n}^{t_k^n} (\Theta_i^0)^T \phi(X_s) ds 
\right)
\left(
W_{t_k^n}-W_{t_{k-1}^n}
\right)\\
&& 
+ \frac{(\sigma_i^0)^2}{n \Delta_n} \sum_{k=1}^n (W_{t_k^n} - W_{t_{k-1}^n})^2.
\end{eqnarray*}
Thus we have that
\[\hat{\sigma}_i^2 - (\sigma_i^0)^2
= (I) + (II) + (III),\]
where
\[
(I)
= \frac{1}{n \Delta_n} \sum_{k=1}^n \left| \int_{t_{k-1}^n}^{t_k^n} (\Theta_i^0)^T \phi(X_s) ds \right|^2,
\]
\[
(II)
= \frac{2 \sigma_i^0}{n \Delta_n} \sum_{k=1}^n 
\left(
\int_{t_{k-1}^n}^{t_k^n} (\Theta_i^0)^T \phi(X_s) ds 
\right)
\left(
W_{t_k^n}-W_{t_{k-1}^n}
\right)
\]
and
\[
(III)
=
\frac{(\sigma_i^0)^2}{n\Delta_n} \sum_{k=1}^n \{(W_{t_k^n}^i - W_{t_{k-1}^n}^i)^2 - \Delta_n\}.
\]
Using Markov's inequalty and Schwartz's inequality, we can evaluate $(I)$ for every $\delta>0$
uniformly in $i$ as follows
\begin{eqnarray*}
\lefteqn{
P \left(
\sup_{1 \leq i \leq p_n} \frac{1}{n \Delta_n} \sum_{k=1}^n 
\left|
\int_{t_{k-1}^n}^{t_k^n} (\Theta_i^0)^T \phi(X_s) ds
\right|^2
\geq \delta
\right)}\\
&\leq&
\frac{1}{n \Delta_n \delta} \sum_{k=1}^n E\left[
\sup_{1 \leq i \leq p_n}
\left|
\int_{t_{k-1}^n}^{t_k^n} (\Theta_i^0)^T \phi(X_s) ds
\right|^2
\right] \\
&\leq&
\frac{1}{n \Delta_n \delta} \sum_{k=1}^n E \left[
\sup_{1 \leq i \leq p_n}
\Delta_n \int_{t_{k-1}^n}^{t_k^n} \left| (\Theta_i^0)^T \phi(X_s) \right|^2 ds
\right]\\
&\leq&
\frac{1}{n \delta} \sum_{k=1}^n \int_{t_{k-1}^n}^{t_k^n} E \left[
\sup_{1 \leq i \leq p_n} \|\Theta_i^0\|_1^2 \max_{l \in T_0^i} |\phi_l (X_s^l)|^2
\right]ds \\
&\leq&
\frac{1}{\delta} \sup_{1 \leq i \leq p_n} \|\Theta_i^0\|_1^2 S^* \tilde{C}_2 \Delta_n^2.
\end{eqnarray*}
The right-hand side of this inequality converges to $0$ if we put $\delta = \Delta_n$. 
This yields that $(I) \rightarrow^p 0$ uniformly in $i$. 

Using Markov's inequality, Schwartz's inequality and Orlitz norm $\|\cdot\|_{\Phi_2}$ with respect to the function 
$\Phi_2(x) = e^{x^2}-1$, we can evaluate $(II)$ for every $\delta > 0$ uniformly in $i$ as follows
\begin{eqnarray*}
\lefteqn{
P \left(
\sup_{1 \leq i \leq p_n} \frac{2 \sigma_i^0}{n \Delta_n} 
\sum_{k=1}^n \left|
\left(
\int_{t_{k-1}^n}^{t_k^n} (\Theta_i^0)^T \phi(X_s) ds
\right)
\left(
W_{t_k^n}^i - W_{t_{k-1^n}}^i
\right)
\right|
\geq \delta
\right)}\\
&\leq&
\frac{2\sup_i \sigma_i^0}{n \Delta_n \delta} \sum_{k=1}^n 
E \left[
\sup_{1 \leq i \leq p_n} \left|
\left(
\int_{t_{k-1}^n} (\Theta_i^0)^T \phi(X_s) ds
\right)
\left(
W_{t_k^n}^i - W_{t_{k-1}^n}^i
\right)
\right|
\right]\\
&\leq&
\frac{2\sup_i \sigma_i^0}{n \Delta_n \delta} \sum_{k=1}^n 
\left(
E \left[
\sup_{1 \leq i \leq p_n} \left|
\int_{t_{k-1}^n} (\Theta_i^0)^T \phi(X_s) ds
\right|^2
\right]
\right)^{\frac{1}{2}}
\left\|
\sup_{1 \leq i \leq p_n}
W_{t_k^n}^i - W_{t_{k-1}^n}^i
\right\|_{L^2} \\
&\leq&
\frac{2\sup_i \sigma_i^0}{n \Delta_n \delta} \sum_{k=1}^n 
(S^* \tilde{C}_2 \Delta_n^2)^{\frac{1}{2}}
K \log(1+p_n) \sup_{1 \leq i \leq p_n} 
\left\|
W_{t_k^n}^i - W_{t_{k-1}^n}^i
\right\|_{\Phi_2} \\
&\leq&
\frac{2\sup_i \sigma_i^0}{n \Delta_n \delta} \sum_{k=1}^n 
(S^* \tilde{C}_2 \Delta_n^2)^{\frac{1}{2}}
K \log(1+p_n) \left(
\frac{8 \Delta_n}{3}
\right)^{\frac{1}{2}},
\end{eqnarray*}
where $K$ is a positive constant which does not depend on $n$. 
If we put $\delta = \Delta_n^{1/3}$, the right-hand side of this inequality converges to $0$. 
So we have that $(II) \rightarrow^p 0$ uniformly in $i$.

$(III)$ is a terminal value of $\{\mathcal{F}_{t_k^n}\}_{k\geq0}$-martingale. 
We apply Bernstein's inequality for martingales (See \cite{key geer}, Lemma 8.9.) to the following 
process: 
\[M_n^i := \sum_{k=1}^n \{(W_{t_k^n}^i-W_{t_{k-1}^n}^i)^2 -\Delta_n\}.\]
To do this, we shall evaluate the next moment for every integer $m \geq 2$,  
\[\frac{1}{n} \sum_{k=1}^n E[|(W_{t_k^n}^i-W_{t_{k-1}^n}^i)^2 - \Delta_n|^m | \mathcal{F}_{t_{k-1}^n}].\]
Noting that $W_{t_k^n}^i-W_{t_{k-1}^n}^i$ is independent of $\mathcal{F}_{t_{k-1}^n}$, 
we have that  
\begin{eqnarray*}
\lefteqn{
\frac{1}{n} \sum_{k=1}^n E[|(W_{t_k^n}^i-W_{t_{k-1}^n}^i)^2 - \Delta_n|^m | \mathcal{F}_{t_{k-1}^n}]} \\
&=& \frac{1}{n} \sum_{k=1}^n E[|(W_{t_k^n}^i-W_{t_{k-1}^n}^i)^2 - \Delta_n|^m] \\
&\leq& \frac{1}{n} \sum_{k=1}^n \sum_{r=0}^m \binom{m}{r} \Delta_n^{m-r} E[(W_{t_k^n}^i-W_{t_{k-1}^n}^i)^{2r}] \\
&=& \Delta_n^m + \sum_{r=1}^m \binom{m}{r} \Delta_n^{m-r} (2r-1)!! \Delta_n^r \\
&=& \Delta_n^m +  \sum_{r=1}^m \frac{(2r-1)!!}{r! (m-r)!} m! \Delta_n^m \\
&<& \sum_{r=0}^m 2^r m! \Delta_n^m \\
&<& \frac{m!}{2} (2 \Delta_n)^{m-2} 4 \Delta_n^2.
\end{eqnarray*}
So it follows from Bernstein's inequality that for every $\epsilon >0$, 
\[P(|M_n^i| \geq \epsilon) \leq 2 \exp \left(-\frac{\epsilon^2}{2(2\Delta_n \epsilon + 4n \Delta_n^2)}\right).\]
We write $\|\cdot\|_{\Phi_1}$ for the Orlicz norm with respect to $\Phi_1(x) :=e^x-1$.
Using Lemma 2.2.10 from 
\cite{key van der vaart} to deduce that there exist a constant
$L_1 >0$ depending only on $\Phi_1$ such that  
\[\left\|\sup_{1 \leq i \leq p_n} |M_n^i| \right\|_{\Phi_1} 
\leq L_1 \{2 \Delta_n \log(1+p_n) + \sqrt{4n\Delta_n^2 \log(1+p_n)}\}.\]
So we obtain from Markov's inequality that 
\[P\left(\sup_{1 \leq i \leq p_n} |M_n^i| \geq \epsilon\right)
\leq \Phi_1\left(\frac{\epsilon}{L_1 \{2 \Delta_n \log(1+p_n) + \sqrt{4n\Delta_n^2 \log(1+p_n)}\}}\right)^{-1}.\]
If we put $\epsilon = \log (1+p_n)$, then the right-hand side of above inequality converges to $0$. 
Note that 
\begin{eqnarray*}
P\left(\sup_{1 \leq i \leq p_n} 
\left|\frac{(\sigma_i^0)^2}{n\Delta_n} \sum_{k=1}^n \{(W_{t_k^n}^i - W_{t_{k-1}^n}^i)^2 - \Delta_n\}\right| 
\geq \frac{(\sigma_i^0)^2\epsilon}{n\Delta_n}\right)
= P\left(\sup_{1 \leq i \leq p_n} |M_n^i| \geq \epsilon\right).
\end{eqnarray*}
If we take $\epsilon = \log(1+ p_n)$, it holds under 
Assumption \ref{regularity} that 
$(\sigma_i^0)^2 \epsilon/n\Delta_n \rightarrow 0$, which yields the conclusion.
\qed
\end{proof}
Note that Theorem \ref{diffusion} and Assumption \ref{regularity} imply that there exists a constant
$\tilde{K}_1$ such that 
\[\lim_{n \rightarrow \infty}P\left(\sup_{1 \leq i \leq p_n} \hat{\sigma}_i^{-2} \geq \tilde{K}_1\right) = 0.\]
\section{Estimators for drift coefficients}
In this section, we define the estimator of $\Theta_i$ by plugging $\hat{\sigma}_i$ in 
quasi-$\log$-likelihood $l_n$.
Hereafter, we write $\psi_n(\Theta_i)$ for the gradient of $l_n(\Theta_i,\hat{\sigma}_i)$ with respect to $\Theta_i$, and $V_n^i$ for Hessian of $-l_n(\Theta_i,\hat{\sigma}_i)$,
$i.e.$,
\[\psi_n(\Theta_i) := \frac{1}{n \Delta_n \hat{\sigma}_i^2} \sum_{k=1}^n 
\phi(X_{t_{k-1}^n}) (X_{t_k^n}^i - X_{t_{k-1}^n}^i - \Theta_i^T \phi(X_{t_{k-1}^n}) \Delta_n)\]
\[V_n^i := \frac{1}{n \hat{\sigma}_i^2} \sum_{k=1}^n \phi(X_{t_{k-1}^n}) \phi(X_{t_{k-1}^n})^T.\]
Note that Hessian matrix does not depend on $\Theta$.
Define the Dantzig selector type estimator $\hat{\Theta}_{n,i}$ of $\Theta_i$ by 
\[\hat{\Theta}_{n,i} := \hat{\Theta}_i := \argmin_{\Theta_i \in \mathcal{C}_n^i} \|\Theta_i\|_1, \quad 
\mathcal{C}_n^i := \{\Theta_i \in \mathbb{R}^{p_n} : \|\psi_n(\Theta_i)\|_\infty \leq \gamma_n\},\]
where $\gamma_n$ is a tuning parameter. 
The goal of this section is to prove the consistency of $\hat{\Theta}_i$.
\subsection{Some discussions on the gradient}
Here, we'll prove that
\[\sup_{1 \leq i \leq p_n}\|\psi_n(\Theta_i^0)\|_\infty \leq \gamma_n\] 
with large probability. 
To do this, we decompose that 
\[\psi_n^j(\Theta_i^0) = A_n^{i,j} + B_n^{i,j},\]
where 
\[A_n^{i,j} := \frac{1}{n \Delta_n \hat{\sigma}_i^2} 
\sum_{k=1}^n \phi_j(X_{t_{k-1}^n}^j) \int_{t_{k-1}^n}^{t_k^n} (\Theta_i^0)^T(\phi(X_s)-\phi(X_{t_{k-1}^n})) ds, 
\]
and
\[B_n^{i,j} := \frac{\sigma_i^0}{n \Delta_n \hat{\sigma}_i^2} 
\sum_{k=1}^n \phi_j(X_{t_{k-1}^n}^j)(W_{t_k^n}^i - W_{t_{k-1}^n}^i).\]
Hereafter, we assume that $\gamma_n =  O( \log(1 + p_n^2)/\sqrt{n \Delta_n})$.
\begin{lem}\label{drift1}
Under Assumption \ref{regularity}, it holds that 
\[\lim_{n \rightarrow \infty}P\left(\sup_{1 \leq i,j \leq p_n} |A_n^{i,j}| \geq \gamma_n\  {\rm and}\ 
 \sup_{1 \leq i \leq p_n}\hat{\sigma}_i^{-2} \leq \tilde{K}_1 \right) = 0.\]
\end{lem}
\begin{proof}
Using Markov's inequality, we have that
\begin{eqnarray*}
\lefteqn{
P\left(\sup_{1 \leq i,j \leq p_n} |A_n^{i,j}| \geq \gamma_n\ {\rm and}\  
\sup_{1 \leq i \leq p_n} \hat{\sigma}_i^{-2} \leq \tilde{K}_1 \right)} \\
&\leq& \frac{\tilde{K}_1}{n \Delta_n \gamma_n} \sum_{k=1}^n
E\left[\sup_{1 \leq i,j \leq p_n} |\phi_j(X_{t_{k-1}^n}^j)| 
\left|\int_{t_{k-1}^n}^{t_k^n} (\Theta_i^0)^T (\phi(X_s)-\phi(X_{t_{k-1}^n})) ds \right|\right] \\
&\leq& \frac{\tilde{K}_1 L}{n \Delta_n \gamma_n} \sum_{k=1}^n
E \left[
\sup_{1 \leq i \leq p_n} \left| 
\int_{t_{k-1}^n}^{t_k^n} (\Theta_i^0)^T (\phi(X_s) - \phi(X_{t_{k-1}^n})) ds \right|
\right] \\
&\leq& \frac{\tilde{K}_1 L}{n \Delta_n \gamma_n} \sum_{k=1}^n
\int_{t_{k-1}^n}^{t_k^n} E \left[
\sup_{1 \leq i \leq p_n} \|\Theta_i^0\|_1 \sup_{l \in T_0^i} 
|\phi_l(X_s^l) - \phi_l(X_{t_{k-1}^n}^l)| 
\right] ds\\
&\leq& \frac{\tilde{K}_1 L L' \sup_{1 \leq i \leq p_n} \|\Theta_i^0\|_1}{n \Delta_n \gamma_n} \sum_{k=1}^n
\int_{t_{k-1}^n}^{t_k^n} E\left[
\sup_{l \in T_0^i} |X_s^l- X_{t_{k-1}^n}^l|
\right] ds \\
&\leq& \frac{\tilde{K}_1 L L' \sup_{1 \leq i \leq p_n} \|\Theta_i^0\|_1}{n \Delta_n \gamma_n}
\cdot n \cdot S^* \Delta_n^{\frac{3}{2}}.
\end{eqnarray*}
The right-hand side of the above inequality converges to $0$ under our assumptions.
So we obtain the conclusion.
\qed
\end{proof}
\begin{lem}\label{drift3}
Under Assumption \ref{regularity}, it holds that 
\[\lim_{n \rightarrow \infty}P\left(\sup_{1 \leq i,j \leq p_n} |B_n^{i,j}| \geq \gamma_n \ {\rm and}\  
\sup_{1 \leq i \leq p_n} \hat{\sigma}_i^{-2} \leq \tilde{K}_1\right) = 0.\]
\end{lem}
\begin{proof}
We apply Bernstein's inequality for martingales to the following terminal value of martingale : 
\[\tilde{M}_n^{i,j} = \sum_{k=1}^n \phi_j(X_{t_{k-1}^n}^j) (W_{t_k^n}^i-W_{t_{k-1}^n}^i).\]
For all integers $m \geq 2$, it holds that 
\begin{eqnarray*}
\lefteqn{
\frac{1}{n} \sum_{k=1}^n 
E\left[
|\phi_j(X_{t_{k-1}^n}^j)|^m |W_{t_k^n}^i-W_{t_{k-1}^n}^i|^m | \mathcal{F}_{t_{k-1}^n}
\right]} \\
&=&
\frac{1}{n} \sum_{k=1}^n |\phi_j(X_{t_{k-1}^n}^j)|^m E[|W_{t_k^n}^i - W_{t_{k-1}^n}^i|^m] \\
&\leq& L^m 
\Delta_n^\frac{m}{2} \frac{2^{\frac{m}{2}}\Gamma(\frac{m+1}{2})}{\pi^{\frac{1}{2}}} \\
&\leq& \frac{m!}{2} (L \sqrt{2 \Delta_n})^{m-2} L^2 (2 \Delta_n).
\end{eqnarray*}
Put 
\[K := L \sqrt{2 \Delta_n},\quad R^2 := L^2 (2 \Delta_n).\]
It follows from Bernstein's inequality that for all $\epsilon >0$ 
\[P(|\tilde{M}_n^{i,j}| \geq \epsilon) \leq 2 \exp \left(- \frac{\epsilon^2}{2(\epsilon K + nR^2)}\right).\]
Using Lemma 2.2.10 from 
\cite{key van der vaart}, 
we have that there exists a constant $L_2 >0$ depending only on $\Phi_1$ such that 
\[\left\|\sup_{1 \leq i,j \leq p_n} |\tilde{M}_n^{i,j}| \right\|_{\Phi_1}
\leq L_2 \{K \log(1+p_n^2) + \sqrt{nR^2 \log(1+p_n^2)}\}.
\]
Using Markov's inequality for $\epsilon = n \Delta_n \gamma_n/(\sigma_i^0 \tilde{K}_1)$, we obtain that 
\begin{eqnarray*}
\lefteqn{
P\left(\sup_{1 \leq i,j \leq p_n} |B_n^{i,j}| \geq \gamma_n \ {\rm and}\  
\sup_{1 \leq i \leq p_n}\hat{\sigma}_i^{-2} \leq \tilde{K}_1\right)} \\
&\leq&
P\left(\sup_{1 \leq i,j \leq p_n} |\tilde{M}_n^{i,j}| \geq \epsilon\right) \\
&\leq&
\Phi_1\left(\frac{\epsilon}{L_2 \{K \log(1+p_n^2) + \sqrt{nR^2 \log(1+p_n^2)}\}} 
\right)^{-1} \\
&\rightarrow& 0.
\end{eqnarray*}
\qed
\end{proof}
Using the above lemmas, we obtain the next theorem.
\begin{thm}
Under Assumption \ref{regularity}, it holds that 
\[\lim_{n \rightarrow \infty}
P\left(
\sup_{1 \leq i \leq p_n} \|\psi_n(\Theta_i^0)\|_\infty \geq 2\gamma_n
\right)
= 0.
\]
\end{thm}
\begin{proof}
It is obvious that Lemma \ref{drift1} and \ref{drift3} imply that 
\[
P\left(
\sup_{1 \leq i \leq p_n} \|\psi_n(\Theta_i^0)\|_\infty \geq 3\gamma_n \ {\rm and}\  
\sup_{1 \leq i \leq p_n} \hat{\sigma}_i^{-2} \leq \tilde{K}_1
\right)
\rightarrow 0.
\]
Noting that  
\begin{eqnarray*}
\lefteqn{
P\left(
\sup_{1 \leq i \leq p_n} \|\psi_n(\Theta_i^0)\|_\infty \geq 3\gamma_n
\right)} \\
&=&
P\left(
\sup_{1 \leq i \leq p_n} \|\psi_n(\Theta_i^0)\|_\infty \geq 3\gamma_n \ {\rm and}\ 
\sup_{1 \leq i \leq p_n} \hat{\sigma}_i^{-2} \leq \tilde{K}_1
\right)\\
&&\ \ \ \ \ +
P\left(
\sup_{1 \leq i \leq p_n} \|\psi_n(\Theta_i^0)\|_\infty \geq 3\gamma_n \ {\rm and}\ 
\sup_{1 \leq i \leq p_n} \hat{\sigma}_i^{-2} \geq \tilde{K}_1
\right)
\end{eqnarray*}
and that
\begin{align*}
P\left(
\sup_{1 \leq i \leq p_n} \|\psi_n(\Theta_i^0)\|_\infty \geq 3\gamma_n \ {\rm and}\  
\sup_{1 \leq i \leq p_n} \hat{\sigma}_i^{-2} \geq \tilde{K}_1
\right)
&\leq P\left(\sup_{1 \leq i \leq p_n}\hat{\sigma}_i^{-2} \geq \tilde{K}_1\right) \rightarrow 0,
\end{align*}
we obtain the conclusion.
\qed
\end{proof}
\subsection{Some discussions on the Hessian}
Define the following factors for $V_n^i$.
\begin{defi}\label{factor for linear}
For every index set $T \subset \{1,\ 2,\ \cdots,\ p_n\}$ and $h \in \mathbb{R}^{p_n}$, 
$h_T$ is a $\mathbb{R}^{|T|}$ dimensional sub-vector of $h$ constructed by extracting the components of $h$ corresponding to the indices in $T$. Define the set $C_T$ by 
\[C_T := \{h \in \mathbb{R}^{p_n} : \|h_{T^c}\|_1 \leq \|h_{T}\|_1\} .\]
We introduce the following factors.
\begin{description}
\item[$(i)$ Compatibility factor]
\[
\kappa(T_0^i,V_n^i) := \inf_{0 \not= h \in C_{T_0^i}} \frac{S_i^{\frac{1}{2}}(h^T V_n^i h)^{\frac{1}{2}}}{\|h_{T_0^i}\|_1}
\]
\item[$(ii)$ Weak cone invertibility factor]
\[
F_q(T_0^i,V_n^i) := \inf_{0 \not= h \in C_{T_0^i}} 
\frac{S_i^{\frac{1}{q}}h^T V_n^i h}{\|h_{T_0^i}\|_1\|h\|_q}, \quad q \in [1,\infty),
\]
\[
F_\infty(T_0^i,V_n^i) := \inf_{0 \not= h \in C_{T_0^i}} 
\frac{(h^T V_n^i h)^{\frac{1}{2}}}{\|h\|_\infty}.
\]
\item[$(iii)$ Restricted eigenvalue]
\[
RE(T_0^i,V_n^i) := \inf_{0 \not= h \in C_{T_0^i}} \frac{(h^T V_n^i h)^{\frac{1}{2}}}{\|h\|_2}.
\]
\end{description}
\end{defi}
We assume that $\kappa(T_0^i,V_n^i)$ satisfies the following condition.
\begin{assumption}\label{matrix2}
For every $\epsilon >0$, there exist $\delta>0$ and $n_0 \in \mathbb{N}$ such that for all $n \geq n_0$, 
\[
P\left(\inf_{1 \leq i \leq p_n} \kappa(T_0^i,V_n^i) >\delta \right) \geq 1-\epsilon.
\]
\end{assumption}
Noting that $\|h_{T_0^i}\|_1^q \geq \|h_{T_0^i}\|_q^q$ for all $q \geq 1$, 
we can see that 
$\kappa(T_0^i;V_n^i)  \leq 2\sqrt{S_i} RE(T_0^i;V_n^i)$, and 
$\kappa(T_0^i;V_n^i) \leq F_q(T_0^i;V_n^i)$. 
So under Assumption \ref{matrix2}, $RE(T_0^i;V_n^i)$ and $F_q(T_0^i;V_n^i)$ also satisfy 
the corresponding conditions.
See 
\cite{key buhlmann} for details of the matrix conditions to deal with 
the sparsity.
\subsection{The consistency of the drift estimator}
The next theorems give the $l_q$ consistency of $\hat{\Theta}_i$ uniformly in $i$ for every $q \in [1,\infty]$.
\begin{thm}\label{l_2}
Under Assumption \ref{regularity} and \ref{matrix2}, the following (i) and (ii) hold true. 
\begin{description}
\item[$(i)$]
It holds that 
\[
\lim_{n \rightarrow \infty}P\left(\sup_{1 \leq i \leq p_n} \|\hat{\Theta}_i - \Theta_i^0\|_2^2
\geq \frac{4 \sup_{1 \leq i \leq p_n} \|\Theta_i^0\|_1 \gamma_n}{\inf_{1 \leq i \leq p_n} RE^2(T_0^i,V_n^i)}
\right) = 0.
\]
In particular, it holds that $\sup_{1 \leq i \leq p_n}\|\hat{\Theta}_i - \Theta_i^0\|_2 \rightarrow^p 0$. 
\item{$(ii)$}
It holds that
\[
\lim_{n \rightarrow \infty}P\left(\sup_{1 \leq i \leq p_n} \|\hat{\Theta}_i - \Theta_i^0\|_\infty^2
\geq \frac{4 \sup_{1 \leq i \leq p_n} \|\Theta_i^0\|_1 \gamma_n}{\inf_{1 \leq i \leq p_n} F_\infty^2(T_0^i,V_n^i)}
\right) = 0.
\]
In particular, it holds that $\sup_{1 \leq i \leq p_n}\|\hat{\Theta}_i - \Theta_i^0\|_\infty \rightarrow^p 0$. 
\end{description}
\end{thm}
\begin{proof}
It is sufficient that 
$\sup_{1 \leq i \leq_n} \|\psi_n(\Theta_i^0)\|_\infty \leq \gamma_n$ implies that 
\[\sup_{1 \leq i \leq p_n} \|\hat{\Theta}_i - \Theta_i^0\|_2^2
\leq \frac{4 \sup_{1 \leq i \leq p_n} \|\Theta_i^0\|_1 \gamma_n}{\inf_{1 \leq i \leq p_n} RE^2(T_0^i,V_n^i)}.\]
By the definition of $\hat{\Theta}_i$, we have that  
\[
\sup_{1 \leq i \leq p_n} \|\psi_n(\hat{\Theta}_i)\|_\infty \leq \sup_{1 \leq i \leq p_n} \gamma_n = \gamma_n.
\]
It follows from triangle inequality that 
\[\sup_{1 \leq i \leq p_n} \|\psi_n(\hat{\Theta}_i) - \psi_n(\Theta_i^0)\|_\infty \leq 2 \gamma_n.\]
Put $h_i := \hat{\Theta}_i - \Theta_i^0$, then we can show that $h_i \in C_{T_0^i}$ 
by the same way as 
\cite{key Fujimori}. 
Using Taylor expansion, we have that 
\[h_i^T[\psi_n(\hat{\Theta}_i) - \psi_n(\Theta_i^0)] = h_i^T V_n^i h_i.\]
So it holds that 
\begin{align*}
\sup_{1 \leq i \leq p_n} h_i^T V_n^i h_i
&= \sup_{1 \leq i \leq p_n}  h_i^T [\psi_n(\hat{\Theta}_i) - \psi_n(\Theta_i^0)] \\
&\leq \sup_{1 \leq i \leq p_n} \|h_i\|_1 \|\psi_n(\hat{\Theta}_i) - \psi_n(\Theta^0_i)\|_\infty \\
&\leq 4 \sup_{1 \leq i \leq p_n} \|\Theta_i^0\|_1 \|\psi_n(\Theta_i^0)\|_\infty \\
&\leq 4 \sup_{1 \leq i \leq p_n} \|\Theta_i^0\|_1 \gamma_n.
\end{align*}
By the definition of $RE(T_0^i,V_n^i)$, we obtain that  
\begin{align*}
RE^2(T_0^i,V_n^i) \|h_i\|_2^2 &\leq h_i^T V_n^i h \\
\sup_{1 \leq i \leq p_n} RE^2(T_0^i,V_n^i) \|h_i\|_2^2 
&\leq 4 \sup_{1 \leq i \leq p_n} \|\Theta_i^0\|_1 \gamma_n \\
\sup_{1 \leq i \leq p_n} \|\hat{\Theta}_i - \Theta_i^0\|_2^2 
&\leq  \frac{4 \sup_{1 \leq i \leq p_n} \|\Theta_i^0\|_1 \gamma_n}{\inf_{1 \leq i \leq p_n} RE^2(T_0^i,V_n^i)}.
\end{align*}
This yields our conclusion in (i).
Using the factor $F_\infty(T_0^i,V_n^i)$ in place of $RE(T_0^i,V_n^i)$, 
we obtain the conclusion in (ii) by the similar way.
\qed
\end{proof}
\begin{thm}\label{l_q}
Under Assumption \ref{regularity} and \ref{matrix2}, the following (i) and (ii) hold true.
\begin{description}
\item[$(i)$]
It holds that 
\[
\lim_{n \rightarrow \infty}P\left(\sup_{1 \leq i \leq p_n} \|\hat{\Theta}_i - \Theta_i^0\|_1
\geq \frac{8 S^* \gamma_n}{\inf_{1 \leq i \leq p_n} \kappa^2(T_0^i,V_n^i)}
\right) = 0.
\]
In particular, it holds that $\sup_{1 \leq i \leq p_n}\|\hat{\Theta}_i - \Theta_i^0\|_2 \rightarrow^p 0$.
\item[$(ii)$]
It holds for every $q \in [1,\infty)$ that 
\[
\lim_{n \rightarrow \infty}P\left(\sup_{1 \leq i \leq p_n} \|\hat{\Theta}_i - \Theta_i^0\|_q
\geq \frac{4 S^{*\frac{1}{q}} \gamma_n}{\inf_{1 \leq i \leq p_n} F_q(T_0^i,V_n^i)}
\right) = 0.
\]
In particular, it holds that $\sup_{1 \leq i \leq p_n} \|\hat{\Theta}_i - \Theta^0_i\|_q \rightarrow^p 0$.
\end{description}
\end{thm}
\begin{proof}
It follows from the proof of Theorem \ref{l_2} that 
\[
\sup_{1 \leq i \leq p_n} h_i^T V_n^i h_i
\leq 2 \sup_{1 \leq i \leq p_n} \|h_i\|_1 \gamma_n.
\]
So by the definition of $\kappa(T_0^i,V_n^i)$, we have that 
\begin{eqnarray*}
\kappa^2(T_0^i,V_n^i) \sup_{1 \leq i \leq p_n} \|h_i\|^2_1
&\leq& 4 S^* \sup_{1 \leq i \leq p_n} h_i^T V_n^i h_i \\
&\leq& 8S^* \sup_{1 \leq i \leq p_n} \|h_i\|_1 \gamma_n
\end{eqnarray*}
We therefore obtain that 
\[
\sup_{1 \leq i \leq p_n} \|h_i\|_1 \leq \frac{8S^* \gamma_n}{\inf_{1 \leq i \leq p_n} \kappa^2(T_0^i,V_n^i)}.
\]
This yields our conclusion in (i). 

On the other hand, by the definition of the factor $F_q(T_0^i,V_n^i)$, we have that 
\[
F_q(T_0^i,V_n^i) \leq \frac{4 S^{*\frac{1}{q}}\gamma_n}{\|h_i\|_q}.
\]
This yields our conclusion in (ii).
\qed
\end{proof}
\section{Variable selection by the Dantzig selector}
\subsection{Estimator for the support index set of the drift coefficients}
In this subsection, we propose the estimator of the support index set $T_0^i$ of the true value $\Theta_i^0$
as follows.
\[
\hat{T}_n^i := \{j : |\hat{\Theta}_{ij}| > \gamma_n^{\frac{1}{2}}\}.
\]
Then, we can prove that $\hat{T}_n^i = T_0^i$ for sufficiently large $n$ with large probability.
\begin{thm}\label{selection}
Under Assumption \ref{regularity} and \ref{matrix2}, it holds that 
\[
\lim_{n \rightarrow \infty} P\left(\hat{T}_n^i = T_0^i\ {\rm for\ all}\ i \in \{1,2,\ldots,p_n\}\right) = 1.
\]
\end{thm}
\begin{proof}
We have that 
\[
\lim_{n \rightarrow \infty} 
P\left(\sup_{1 \leq i \leq p_n} \|\hat{\Theta}_i - \Theta_i^0\|_\infty > \gamma_n^{\frac{1}{2}}\right) = 0
\]
by Theorem \ref{l_2}.
Therefore, it is sufficient to show that the next inequality
\[
\sup_{1 \leq i \leq p_n} \|\hat{\Theta}_i - \Theta_i^0 \|_\infty \leq \gamma_n^{\frac{1}{2}}
\]
implies that 
\[\hat{T}_n^i = T_0^i,\quad {\rm for\ all}\ i =1,2,\ldots,p_n.\]
For every $j \in T_0^i$, it follows from the triangle inequality that 
\[|\Theta_{ij}^0| - |\hat{\Theta}_{ij}| \leq |\hat{\Theta}_{ij} - \Theta_{ij}^0| \leq \gamma_n^{\frac{1}{2}}.\]
Then, we have that
\[
|\hat{\Theta}_{ij}| \geq |\Theta_{ij}^0|-\gamma_n^{\frac{1}{2}} > \gamma_n^{\frac{1}{2}}
\]
for sufficiently large $n$, which implies that $T_0^i \subset \hat{T}_n^i$ for every $i \in \{1,2,\ldots,p_n\}$. 
On the other hand, for every $j \in T_0^{ic}$, we have that 
\[
|\hat{\Theta}_{ij} - \Theta_{ij}^0| = |\hat{\Theta}_{ij}| \leq \gamma_n
\]
since it holds that $\Theta_{ij}^0 = 0$.
Then, we can see that $j \in \hat{T}_n^{ic}$ which implies that 
$\hat{T}_n^i \subset T_0^i$ for every $i \in \{1,2,\ldots,p_n\}$.
We thus obtain the conclusion.
\qed
\end{proof}
\subsection{New estimator for drift coefficients after variable selection}
We construct the new estimator $\hat{\Theta}_i^{(2)}$ by the solution
to the next equation
\begin{equation}\label{Q-MLE}
\psi_n(\Theta_{i \hat{T}_n^i})_{\hat{T}_n^i} = 0,\quad \Theta_{i \hat{T}_n^{ic}} = 0.
\end{equation}
We will prove the asymptotic normality of the estimator 
$\hat{\Theta}_{i \hat{T}_n^i}^{(2)}$ for every $i \in \{1,2,\ldots,p_n\}$. 
To do so, we define the $S_i \times S_i$ matrix $Q_{T_0^i,T_0^i}^i$ by 
\[
Q_{T_0^i, T_0^i}^i :=
\frac{1}{(\sigma_i^0)^2}\int_{\mathbb{R}^{S_i}} \phi(x)_{T_0^i} \phi(x)_{T_0^i}^T \mu_0^i(dx).
\]
Hereafter, we assume that this matrix $Q_{T_0^i, T_0^i}^i$ is invertible. 
The next lemma states that $V_{n T_0^i,T_0^i}^i$ is approximated by 
$Q_{T_0^i, T_0^i}^i$ with large probability for sufficiently large $n$.
\begin{lem}\label{hessianapp}
Under Assumption \ref{regularity}, the random sequence $\epsilon_n^i$ defined by 
\[
\epsilon_n^i := \|V_{n T_0^i,T_0^i}^i - Q_{T_0^i, T_0^i}^i\|_\infty,\quad i \in \{1,2,\ldots,p_n\}
\]
converges to $0$ in probability.
\end{lem}
\begin{proof}
Note that 
\[
V_{n T_0^i,T_0^i}^i = \frac{1}{n\hat{\sigma}_i^2} \sum_{k=1}^n 
\phi(X_{t_{k-1}^n T_0^i})_{T_0^i} \phi(X_{t_{k-1}^n T_0^i})_{T_0^i}^T.
\]
It holds that 
\[
\epsilon_n^i \leq (I) + (II) + (III),
\]
where 
\[
(I) := \left\|
V_{n T_0^i, T_0^i}^i - \frac{1}{T_n \hat{\sigma}_i^2} 
\int_0^{T_n} \phi(X_{t T_0^i})_{T_0^i} \phi(X_{t T_0^i})_{T_0^i}^T dt
\right\|_\infty,
\]
\[
(II) :=  \left\|
\frac{1}{T_n \hat{\sigma}_i^2} \int_0^{T_n} \phi(X_{t T_0^i})_{T_0^i} \phi(X_{t T_0^i})_{T_0^i}^T dt - \frac{1}{\hat{\sigma}_i^2} 
\int_{\mathbb{R}^{S_i}} \phi(x)_{T_0^i} \phi(x)_{T_0^i}^T \mu_0^i(dx)
\right\|_\infty
\]
and 
\[
(III) := \left\|
\frac{1}{\hat{\sigma}_i^2} \int_{\mathbb{R}^{S_i}} \phi(x)_{T_0^i} \phi(x)_{T_0^i}^T \mu_0^i(dx)
 - Q_{T_0^i, T_0^i}^i
\right\|_\infty.
\]
It is obvious that $(II)$ and $(III)$ are $o_p(1)$ by 
Assumption \ref{regularity} and Theorem \ref{diffusion}.
To complete the proof, it is sufficient to prove that
\[
P \left( \left\|
V_{n T_0^i, T_0^i}^i - \frac{1}{T_n \hat{\sigma}_i^2} 
\int_0^{T_n} \phi(X_{t T_0^i})_{T_0^i} \phi(X_{t T_0^i})_{T_0^i}^T dt
\right\|_\infty
\geq \delta\  {\rm and}\ 
 \sup_{1 \leq i \leq p_n}\hat{\sigma}_i^{-2} \leq \tilde{K}_1
\right) \rightarrow 0
\]
as $n \rightarrow \infty$ for every $\delta >0$. 
Using Markov's inequality, we can see that 
\begin{eqnarray*}
\lefteqn{
P \left(\left\|
V_{n T_0^i,T_0^i}^i
- \frac{1}{T_n \hat{\sigma}_i^2} 
\int_0^{T_n} \phi(X_{t T_0^i})_{T_0^i} \phi(X_{t T_0^i})_{T_0^i}^T dt
\right\|_\infty
\geq \delta\  {\rm and}\ 
 \sup_{1 \leq i \leq p_n}\hat{\sigma}_i^{-2} \leq \tilde{K}_1
\right)} \\
&\leq&
\frac{\tilde{K}_1}{n \Delta_n \delta} \sum_{k=1}^n 
\int_{t_{k-1}^n}^{t_k^n} E \left[\sup_{j,l \in T_0^i} 
|\phi_j(X_t^j) \phi_l(X_t^l) - \phi_j(X_{t_{k-1}^n}^j) \phi_l(X_{t_{k-1}^n}^l)|
\right] dt.
\end{eqnarray*}
Moreover, it follows from triangle inequality and Schwartz's inequality that
\begin{eqnarray*}
\lefteqn{
E \left[\sup_{j,l \in T_0^i} 
|\phi_j(X_t^j) \phi_l(X_t^l) - \phi_j(X_{t_{k-1}^n}^j) \phi_l(X_{t_{k-1}^n}^l)|
\right]}\\
&\leq&
E \left[\sup_{j,l \in T_0^i}|\phi_l(X_t^l) (\phi_j(X_t^j) - \phi_j(X_{t_{k-1}^n}^j))|\right]\\
&&
+
E \left[\sup_{j,l \in T_0^i}|\phi_j(X_{t_{k-1}^n}^j)(\phi_l(X_t^l) - \phi_l(X_{t_{k-1}^n}^l))| \right] \\
&\leq&
\left(
E \left[
\sup_{l \in T_0^i} |\phi_l(X_t^l)|^2
\right]
\right)^{\frac{1}{2}}
\left(
E\left[
\sup_{j \in T_0^i} |\phi_j(X_t^j) - \phi_j(X_{t_{k-1}^n}^j)|^2
\right]
\right)^{\frac{1}{2}} \\
&&
+
\left(
E \left[
\sup_{j \in T_0^i} |\phi_j(X_{t_{k-1}^n}^j)|^2
\right]
\right)^{\frac{1}{2}}
\left(
E\left[
\sup_{l \in T_0^i} |\phi_l(X_t^l) - \phi_l(X_{t_{k-1}^n}^l)|^2
\right]
\right)^{\frac{1}{2}} \\
&\leq& 2 S^* L L' \Delta_n^{\frac{1}{2}}.
\end{eqnarray*}
We thus have that 
\begin{eqnarray*}
P \left(\left\|
V_n^i
- \frac{1}{T_n \hat{\sigma}_i^2} \int_0^{T_n} X_t X_t^T dt
\right\|_\infty
\geq \delta\  {\rm and}\ 
 \sup_{1 \leq i \leq p_n}\hat{\sigma}_i^{-2} \leq \tilde{K}_1
\right)
&\leq&
\frac{2\tilde{K}_1 L L' S^*}{n \Delta_n \delta} \cdot  n \Delta_n^{\frac{3}{2}}.
\end{eqnarray*}
If we put $\delta = n^{-\eta}$ for $\eta \in (0, \alpha/2)$, 
then the right-hand-side of this inequality converges to $0$, 
which means that $(I) = o_p(1)$.
\qed
\end{proof}
Now, we are ready to prove the asymptotic normality of $\hat{\Theta}_{i \hat{T}_n^i}^{(2)}$ 
in the following sense.
\begin{thm}
It holds for every $i \in \mathbb{N}$ that  
\[
\sqrt{t_n^n} (\hat{\Theta}_{i \hat{T}_n^i}^{(2)} - \Theta_{i T_0^i}^0)
1_{\{\hat{T}_n^i = T_0^i\}} \rightarrow^d N(0,Q_{T_0^i,T_0^i}^{i -1})
\]
as $n \rightarrow \infty$. Note that for every $i \in \mathbb{N}$, it holds that 
$i < p_n$ for sufficiently large $n$.
\end{thm}
\begin{proof}
Using Taylor expansion, we have that 
\[
\psi_n(\hat{\Theta}_{i \hat{T}_n^i}^{(2)})_{\hat{T}_n^i}
= \psi_n(\Theta_{i \hat{T}_n^i}^0) - V_{n \hat{T}_n^i \hat{T}_n^i}^i 
(\hat{\Theta}_{n \hat{T}_n^i}^{(2)} - \Theta_{i \hat{T}_n^i}^0).
\]
It follows from the definition of the estimator $\hat{\Theta}_i^{(2)}$ that 
\[
\sqrt{t_n^n} V_{n \hat{T}_n^i \hat{T}_n^i}^i  (\hat{\Theta}_{i \hat{T}_n^i}^{(2)} - \Theta_{i T_0^i}^0)
1_{\{\hat{T}_n^i = T_0^i\}}
= \sqrt{t_n^n} \psi_n(\Theta_{i T_0^i}^0)_{T_0^i} 1_{\{\hat{T}_n^i = T_0^i\}}.
\]
We can decompose that $\sqrt{t_n^n} \psi_n(\Theta_{i T_0^i}^0)_{T_0^i} = (I) + (II) + (III)$, where 
\[
(I) = \frac{\sigma_i^0}{\sqrt{n \Delta_n}\hat{\sigma}_i^2}
\sum_{k=1}^n \phi(X_{t_{k-1}^n T_0^i})_{T_0^i} 
\int_{t_{k-1}^n}^{t_k^n} (\Theta_{iT_0^i}^0)^T \{\phi(X_{s T_0^i})_{T_0^i} - \phi(X_{t_{k-1}^n T_0^i})_{T_0^i} \} ds,
\] 
\[(II) =
\left(\frac{\sigma_i^0}{\sqrt{n \Delta_n} \hat{\sigma}_i^2} 
- \frac{1}{\sqrt{n \Delta_n} \sigma_i^0}\right)
\sum_{k=1}^n \phi(X_{t_{k-1}^n T_0^i})_{T_0^i} (W_{t_k^n}^i - W_{t_{k-1}^n}^i)
\]
and 
\[(III) =
\frac{1}{\sqrt{n \Delta_n} \sigma_i^0}
\sum_{k=1}^n \phi(X_{t_{k-1}^n T_0^i})_{T_0^i} (W_{t_k^n}^i - W_{t_{k-1}^n}^i)
\]
We can show that $(I) = o_p(1)$ by the similar way to the proof of Lemma \ref{drift1}.
Next, we will apply the martingale central limit theorem for $(III)$.
Define the martingale differences $\{\xi_k\}_{k=1,2,\ldots,n}$ by
\[
\xi_k := \frac{1}{\sqrt{n \Delta_n} \sigma_i^0}  \phi(X_{t_{k-1}^n T_0^i})_{T_0^i} (W_{t_k^n}^i - W_{t_{k-1}^n}^i).
\]
It holds for every $j,l \in T_0^i$ that
\begin{eqnarray*}
\lefteqn{
\sum_{k=1}^n E\left[
\frac{1}{\sqrt{n \Delta_n (\sigma_i^0)^2}}
\phi_j(X_{t_{k-1}^n}^j) \phi_l(X_{t_{k-1}^n}^l) | \mathcal{F}_{t_{k-1}^n}
\right]
}\\
&=&
\frac{1}{n \Delta_n (\sigma_i^0)^2}
\sum_{k=1}^n 
\phi_j(X_{t_{k-1}^n}^j) \phi_l(X_{t_{k-1}^n}^l) E[(W_{t_k^n}^i - W_{t_{k-1}^n}^i)^2]\\
&=&
\frac{1}{n (\sigma_i^0)^2}
\sum_{k=1}^n \phi_j(X_{t_{k-1}^n}^j) \phi_l(X_{t_{k-1}^n}^l).
\end{eqnarray*}
We can see that right-hand side converges to the $(j,l)$-component of the matrix $Q_{T_0^i T_0^i}^i$ in probability
by the same way of the proof of Lemma \ref{hessianapp}.
Moreover, we can check the Lyapnov's condition: 
\[
\sum_{k=1}^n E \left[
\|\xi_k\|_2^{2 + \delta} | \mathcal{F}_{t_{k-1}^n}
\right] \rightarrow^p 0
\]
for $\delta = 2$, which implies Lindeberg's condition: 
\[
\sum_{k=1}^n E \left[
\|\xi_k\|_2^2 1_{\{\|\xi_k\|_2 > \epsilon\}} | \mathcal{F}_{t_{k-1}^n}
\right] \rightarrow^p 0,\
\]
for every $\epsilon > 0$.
Then, we obtain that
\[
\frac{1}{\sqrt{n \Delta_n} \sigma_i^0}
\sum_{k=1}^n \phi(X_{t_{k-1}^n T_0^i})_{T_0^i} (W_{t_k^n}^i - W_{t_{k-1}^n}^i)
\rightarrow^d N(0, Q_{T_0^i T_0^i}^i)
\]
by martingale central limit theorem.
Noting that 
\[
(II) = 
\left(\frac{(\sigma_i^0)^2}{\hat{\sigma}_i^2}-1
\right) \frac{1}{\sqrt{n \Delta_n} \sigma_i^0}
\sum_{k=1}^n \phi(X_{t_{k-1}^n T_0^i})_{T_0^i} (W_{t_k^n}^i - W_{t_{k-1}^n}^i)
\]
and $(III) = O_p(1)$,
we obtain that $(II) = o_p(1)$ since $\hat{\sigma}_i$ is a consistent estimator for $\sigma_i^0$.
Using the above results and Lemma \ref{hessianapp}, we have that 
\[
\sqrt{t_n^n} (\hat{\Theta}_{i \hat{T}_n^i}^{(2)} - \Theta_{i T_0^i}^0)
1_{\{\hat{T}_n^i = T_0^i\}}
= Q_{T_0^i T_0^i}^{i -1} \sqrt{t_n^n} \psi_n(\Theta_{i T_0^i}^0)_{T_0^i} 1_{\{\hat{T}_n^i = T_0^i\}}
+ o_p(1).
\]
Since it holds that $1_{\{\hat{T}_n^i = T_0^i\}} \rightarrow^p 1$ by Theorem \ref{selection}, we can use Slutsky's theorem to derive our conclusion.
\qed
\end{proof}
\section{Concluding remarks}
In summary, we can construct the asymptotically good estimator for our model even in high-dimensional settings if the sparsity of the parameter is fixed or bounded. 
If the sparsity is not bounded, we may not reduce the dimension of the parameter.
In such cases, the asymptotically normal estimator can not be constructed by the equation 
(\ref{Q-MLE}).

In this paper, we assume that the diffusion coefficients $\sigma_i$'s are constants.
However, it may be possible to consider the case when each $\sigma_i$ has more
complicated structures. For example, we can consider the following model:
\[
X_t^i = X_0^i + \int_0^t \Theta_i^T \phi(X_s) ds + \int_0^t \exp(\beta_i^T \varphi(X_s)) dW_s^i, 
\quad i = 1,2,\ldots,p,
\]
where $\beta_i \in \mathbb{R}^p$ and $\varphi(\cdot)$ is an appropriate smooth function.
According to 
\cite{fujimori-nishiyama2}, we can construct estimators for $\beta_i$ by 
the Dantzig selector and prove the $l_q$ consistency of the estimators for every $q \in [1,\infty].$ Therefore, we may prove the same asymptotic properties of $\Theta$ even for the 
above model which has high-dimensional parameters in diffusion coefficients. 

Besides, variable selection consistency of the estimator of drift matrix is important for 
applications such as graphical modeling which can be seen in 
\cite{wainwright} 
and 
\cite{periera-ibrahimi}.
In future, we would like to consider such applications and present some numerical results. 
\vskip 20pt
{\bf Acknowledgements.}
The author would like to express the appreciation to Prof.\ Y.\ Nishiyama of Waseda University and Dr.\ K.\ Tsukuda of the University of Tokyo for long hours discussions about this paper. 


\begin{thebibliography}{00}
\bibitem{key bickel-ritov-tsybakov} 
Bickel, P.J., Ritov, Y. and Tsybakov, A.B. (2009). 
Simultaneous analysis of lasso and Dantzig selector. {\it  Ann. Statist.} 37, no. 4, 1705-1732.
\bibitem{key tao} 
Cand\'{e}s, E. and Tao, T. (2007). 
The Dantzig selector: statistical estimation when $p$ is much larger than $n$. {\it Ann. Statist}. 35, no.6, 2313-2351.
\bibitem{key Fujimori} Fujimori, K. and Nishiyama, Y. (2017 a). The $l_q$ consistency of the Dantzig selector for 
Cox's proportional hazards model. {\it J. Statist. Plann. Inference} 181, 62-70.
\bibitem{fujimori-nishiyama2} Fujimori, K. and Nishiyama, Y. (2017 b).
The Dantzig selector for diffusion processes with covariates.
{\it J. Japan Statist. Soc.}\ 47, no.1, 59-73.
\bibitem{key jacod} Genon-Catalot, V. and Jacod, J. (1993). On the estimation of the diffusion coefficient for 
multi-dimensional diffusion processes. {\it Ann. Inst. H. Poincar\'e Probab. Statist.} 29, no.1, 119-151.
\bibitem{iacus} Gregorio, A. D. and Iacus, S. M. (2012).
Adaptive LASSO-type estimation for multivariate diffusion processes.
{\it Econometric Theory} 28, no.4, 838-860.
\bibitem{key kessler97} Kessler, M. (1997). Estimation of an ergodic diffusion from discrete observations. 
{\it Scand. J. Statist.} 24, no.2, 211-229.
\bibitem{key masuda-shimizu} Masuda, H. and Shimizu, Y. (2017). 
Moment convergence in regularized estimation
under multiple and mixed-rates asymptotics.
{\it Mathematical Methods of Statistics} no.2, 81-110.
\bibitem{periera-ibrahimi} Periera, J. B. A. and Ibrahimi, M. (2014).
Support recovery for the drift coefficient of high-dimensional diffusions. 
{\it IEEE Trans. Inform. Theory} 60, no.7, 4026-4049. 
\bibitem{wainwright} Ravikumar, P., Wainwright, M. J. and Lafferty, J. D. (2010).
High-dimensional Ising model selection using $l_1$-regularized logistic regression.
{\it Ann. Statist.} 38, no.3, 1287-1319.
\bibitem{key tibshirani96}
Tibshirani, R. (1996). 
Regression shrinkage and selection via the Lasso. 
{\it J. Roy. Statist. Soc. Ser. B} 58, no.1, 267-288.
\bibitem{key geer}
van de Geer, S. A. (2000)
{\it Empirical Processes in M-Estimation.}
Cambridge Series in Statistical and Probabilistic Mathematics, 6.
\bibitem{key buhlmann}
van de Geer, S.A. and B\"{u}hlmann, P. (2009). 
On the conditions used to prove oracle results for the Lasso. {\it Electron. J. Stat}. 3, 1360-1392.
\bibitem{key van der vaart}
van der Vaart, A.W. and Wellner, J.A. (1996). 
{\it Weak Convergence and Empirical Processes. With Applications to Statistics.} Springer Series in Statistics. Springer-verlag, New York.
\bibitem{yoshida92} Yoshida, N. (1992). Estimation for diffusion processes from discrete observation. 
{\it J. Multivariate Anal.} 41, no.2, 220-242.
\end{thebibliography}
\end{document}